\documentclass[10pt]{amsart}

\usepackage{hyperref}
\usepackage{amsmath}
\usepackage{graphicx}
\usepackage{tikz}
\usepackage{tikz-cd}
\usepackage{tikz-3dplot}
\usepackage{subcaption}
\usepackage{enumitem}

\usetikzlibrary{snakes}
\usepackage[all]{xy}
\usepackage{framed}
\usepackage{mathrsfs}%For mathscr
\usepackage{amssymb}

\usepackage{colonequals} %\colonequal

\theoremstyle{plain}
\newtheorem{theorem}{Theorem}[section]
\newtheorem{lemma}[theorem]{Lemma}
\newtheorem{proposition}[theorem]{Proposition}
\newtheorem{corollary}[theorem]{Corollary}

\theoremstyle{definition}
\newtheorem{definition}[theorem]{Definition}
\newtheorem{example}[theorem]{Example}
\newtheorem{remark}[theorem]{Remark}

\numberwithin{equation}{section}

\newcommand{\QQ}{\mathbb{Q}}
\newcommand{\RR}{\mathbb{R}}
\newcommand{\ZZ}{\mathbb{Z}}
\newcommand{\NN}{\mathbb{N}}
\newcommand{\PP}{\mathbb{P}}

\DeclareMathOperator{\Pic}{Pic}

\DeclareMathOperator{\GL}{GL}
\DeclareMathOperator{\wSR}{wSR}
\DeclareMathOperator{\SR}{SR}
\DeclareMathOperator{\spa}{span}
\DeclareMathOperator{\CDiv}{CDiv}
\DeclareMathOperator{\Div}{Div}
\DeclareMathOperator{\clas}{Cl}

\begin{document}

\title[Integral bases]{Integral bases for the second degree cohomology of 4-dimensional toric orbifolds}

\author[T. So]{Tseleung So}
\address{Institute of Mathematical Science, Pusan National University, Busan 46241, Republic of Korea}
\email{larry.so.tl@gmail.com}

\author[J. Song]{Jongbaek Song}
\address{Department of Mathematics Education, Pusan National University, Busan, Republic of Korea}
\email{jongbaek.song@pusan.ac.kr}

% \date{\today}
\thanks{The authors are supported by the National Research Foundation of Korea (NRF) grant funded by the Korea government(MSIT) (RS-2025-00555914). }



\maketitle 
\abstract 
We study toric orbifolds of real dimension four with vanishing odd-degree cohomology and obtain a basis for its degree-two equivariant cohomology with integral coefficients by identifying it with the intersection of certain lattices. As applications, we provide an alternative construction of the \emph{algebraic cellular basis} for integral ordinary cohomology \cite{FSS2}. In addition, when the toric orbifold is an algebraic variety, we determine its Cartier divisor group and Picard group.
\endabstract

\section{Introduction}

The integral cohomology of a toric orbifold captures information about its singularities that is often invisible in rational cohomology. For instance, the ordinary and equivariant cohomology with rational coefficients of weighted projective spaces~$\mathbb{P}^n(\chi)$ are isomorphic to those of projective spaces $\PP^n$. In contrast,  Kawasaki \cite{Ka} showed that the multiplication structure of the integral cohomology of~$\PP^n(\chi)$ depends on its weight vector $\chi \in \NN^{n+1}$, which encodes the orbifold singularities of the space. Moreover, determining generators for its integral equivariant cohomology also involves nontrivial arithmetic calculations depending on the weight vector \cite{BFR}. In this sense, the integral (equivariant) cohomology theory for orbifolds is considerably subtler than its rational counterpart. This problem is already non-trivial even in degree $2$, where one seeks explicit integral classes arising from the underlying combinatorial data of a toric orbifold. 

For a toric orbifold $X(P, \lambda)$ associated with combinatorial data $(P, \lambda)$, called a \emph{characteristic pair} (see Section \ref{subsec_toric_orb}), this problem has been partially addressed, in the case where the odd degree cohomology of $X(P, \lambda)$ vanishes, by the \emph{weighted Stanley--Reisner ring} $\wSR(P, \lambda)$ introduced in \cite{BSS, DKS}. It gives an algebraic model for the integral equivariant cohomology ring of such a toric orbifold. By definition, $\wSR(P,\lambda)$ is a subring of the Stanley--Reisner ring $\SR(P)$, determined by the integrality conditions (see Definitions \ref{def_int_cond} and \ref{def_wSR}). This description of $\wSR(P, \lambda)$ makes it possible to test whether a given element of $\SR(P)$ belongs to $\wSR(P,\lambda)$. However, it does not in general provide explicit generators or bases, so the problem must be treated on a case-by-case basis.

In this paper we focus on the real four dimensional case.
Our main result gives a full description of the degree-two component $\wSR^2(P,\lambda)$ of the weighted Stanley--Reisner ring as follows.

\begin{theorem}[restated in Theorem~\ref{thm_main}]\label{thm_main intro}
Let $X(P,\lambda)$ be a toric orbifold of real dimension four with trivial $H^3(X(P,\lambda);\ZZ)$. As $\ZZ$-modules, it follows that 
\[
H^2_{T^2}(X(P,\lambda);\ZZ)\cong\wSR^2(P,\lambda)\cong\ZZ\langle \textbf{a}\rangle\oplus\ZZ\langle\textbf{b}\rangle\oplus\phi(K),
\]
where $\textbf{a},\textbf{b}$ are vectors deduced from $\lambda$, $K$ is the space of some linear relations and  $\phi$ is a linear map (See \eqref{eq_a_b}, \eqref{eq_K} and \eqref{eq_phi_linear_operator}, respectively, for their explicit definitions). 
\end{theorem}

The first isomorphism follows from the results of \cite{BSS,DKS}, whereas the second isomorphism is new, which will be established in Section \ref{sec_lat_int}. Here we emphasize that a basis of $K$ can be obtained by some elementary methods in linear algebra, since $K$ is determined by a system of linear equations over integers. Thus, Theorem~\ref{thm_main intro} yields an explicit basis of $\wSR^2(P,\lambda)$, and hence that of the equivariant cohomology group $ H^2_{T^2}(X(P,\lambda);\ZZ)$.

Two applications of Theorem \ref{thm_main intro} are given in Section~\ref{sect_applications}.
First, we apply it to provide an alternative definition of the \emph{algebraic cellular basis} of~\cite{FSS2}. This is a special additive basis for $\tilde{H}^*(X(P,\lambda);\ZZ)$ constructed by using methods from homotopy theory and combinatorics. Its well-definedness and existence were established through a series of long and technical proofs. Theorem \ref{thm_main intro} redefines the algebraic cellular basis of $X(P, \lambda)$ in a more general setting and provides a simpler proof (see Theorem~\ref{thm_alg cellular basis}).

Second, we consider the case where $X(P, \lambda)$ is a toric variety $X_\Sigma$ associated with a fan $\Sigma$. We  verify that the group of torus-invariant Cartier divisors $\CDiv_T(X_{\Sigma})$ and its Picard group $\Pic(X_{\Sigma})$ are isomorphic to $\wSR^2(P,\lambda)$ and $\phi(K)$, respectively.
In general, computing $\CDiv_T(X_\Sigma)$ and $\Pic(X_\Sigma)$ is subtle, as it requires solving integral Cartier conditions encoded by the fan (see, for example, \cite[Theorem~4.2.8]{CLS} or \cite[Section 3.3]{Ful}). As another application of the main result, 
we obtain a complete description of these groups and compute the index of $\Pic(X_\Sigma)$ in the class group $\clas(X_\Sigma)$.

\section{Toric orbifolds and Weighted Stanley--Reisner rings}\label{sect_preliminary}
\subsection{Toric orbifolds}\label{subsec_toric_orb}
We begin with a brief summary of a toric orbifold from a constructive viewpoint. Let $P$ be an $n$-dimensional simple polytope and let $\mathcal{F}(P)$ denote the set of its facets. For each vertex $v$ of $P$, write $\mathcal{F}_v(P)$ for the subset of~$\mathcal{F}(P)$ consisting of facets that contain $v$.

A \emph{characteristic function} $\lambda\colon \mathcal{F}(P)\to \ZZ^n$ is a map such that
\begin{enumerate}
\item[(i)] each $\lambda(F)$ is primitive, 
\item[(ii)] for every vertex $v$, the set $\{\lambda(F)\mid F\in \mathcal{F}_v(P)\}$ is linearly independent.
\end{enumerate}
The pair $(P,\lambda)$ is called a \emph{characteristic pair}.

Since $P$ is simple, each face $E$ of codimension $k$ is the intersection of exactly $k$ facets, that is, $E=F_{i_1}\cap \dots\cap F_{i_k}$ for some $F_{i_1},\ldots,F_{i_k}\in\mathcal{F}(P)$. This expression is unique up to permutation. Identifying $\ZZ^n$ with the standard lattice in the Lie algebra of the $n$-dimensional compact torus $T^n$, we denote by $T_E$ the subgroup of~$T^n$ generated by $\{\lambda(F_{i_j})\mid j=1, \dots, k\}$. 

We then define the \emph{toric orbifold} associated with $(P,\lambda)$ as the quotient space 
\begin{equation}\label{eq_cont_toric_orb}
X(P, \lambda) \colonequals P\times T^n /_\sim,
\end{equation}
equipped with the canonical $T^n$-action, where $(p,t)\sim (q,s)$ if and only if $p=q$ and $t^{-1}s \in T_E$, with $E$ being the unique face whose relative interior contains $p$. If, in addition, for every vertex $v$, the set $\{\lambda(F)\mid F\in \mathcal{F}_v(P)\}$ forms an integral basis of~$\ZZ^n$, then $X(P, \lambda)$ is a smooth manifold and is called a \emph{quasitoric manifold}. Otherwise, it is an orbifold in general. For a more axiomatic definition of toric orbifolds (and quasitoric manifolds), see \cite{DJ, PS} and \cite{BP-book}. 

\begin{remark}\label{rmk_toric_var}
When $P$ is a simple rational polytope in $M\otimes_\ZZ \RR$, where $M$ is the character lattice of $T^n$, and $\lambda$ is defined by assigning to each facet its primitive outward (or inward) normal vector, one obtains a symplectic toric orbifold or a projective toric variety with orbifold singularities (see, for example, \cite{LeTo} and \cite[Chapters 2,3]{CLS}). It turns out that these spaces are  equivariantly homeomorphic to the toric orbifold $X(P, \lambda)$ defined in \eqref{eq_cont_toric_orb}. The reader is  referred to \cite{Fr} and the references therein. 
\end{remark}

\subsection{Equivariant cohomology of $X(P, \lambda)$}
When $X(P, \lambda)$ is smooth (that is, when $X(P,\lambda)$ is a quasitoric manifold), its integral equivariant cohomology ring $H^\ast_{T^n}(X(P,\lambda);\ZZ)$ is isomorphic to the \emph{Stanley--Reisner ring}
\[
\SR(P) = \ZZ[x_1,\ldots,x_m] / \mathcal{I},
\]
where 
$\deg x_i =2$ and $\mathcal{I}$ is the ideal given by
\begin{equation}\label{eq_ideal I}
\mathcal{I}=\langle x_{i_1}\cdots x_{i_k}\mid F_{i_1}\cap\cdots\cap F_{i_k} = \emptyset \rangle.
\end{equation}

When $X(P,\lambda)$ is not smooth, this isomorphism still holds with rational coefficients, that is,
\begin{equation}\label{eq_rational_cohom_isom}
H^\ast_{T^n}(X(P, \lambda);\QQ) \cong \SR(P)\otimes\QQ.
\end{equation}
However, in general, the integral version of this isomorphism does not exist. Nevertheless, in some case $H^*_{T^n}(X(P,\lambda);\ZZ)$ can be identified with a subring $\wSR(P, \lambda)$ of $\SR(P)$, called the \emph{weighted Stanley--Reisner ring}, introduced in \cite{BSS, DKS}.\footnote{The notion of the weighted Stanley--Reisner ring was first introduced in \cite{BSS} for projective toric orbifolds. Subsequently, the authors of \cite{DKS} extended this concept to the broader class of \emph{torus orbifolds}, referring to it as the \emph{weighted face ring}. In this paper, we continue to use the term \emph{weighted Stanley--Reisner ring}.}
We recall its definition below.

\begin{definition}\label{def_int_cond}
Given a characteristic pair $(P,\lambda)$ and a vertex $v$ of $P$, write $\mathcal{F}_v(P)=\{F_{i_1},\ldots,F_{i_n}\}$ for the set of facets containing $v$. Define an $m$-tuple
\[
\mathbf{z}^v=(z^v_1,\ldots,z^v_m)\in \QQ[u_1,\ldots,u_n]^{\oplus m}
\]
of rational polynomials in $n$ variables as follows:
\begin{itemize}
\item $z^v_j = 0$ for $j\notin \{i_1, \dots, i_n\}$;
\item $\begin{bmatrix} z^v_{i_1} \\ \vdots \\ z^v_{i_n} \end{bmatrix} = 
\left[ \begin{array}{c|c|c} 
& & \\
\lambda(F_{i_1}) & \cdots &  \lambda(F_{i_n}) \\
& & 
\end{array}\right]^{-1}
\cdot \begin{bmatrix} u_1 \\ \vdots  \\ u_n \end{bmatrix}
$.
\end{itemize}
Then, a polynomial $f(x_1, \dots, x_m)\in \ZZ[x_1, \dots, x_m]$ is said to satisfy the \emph{integrality condition} with respect to $(P, \lambda)$ if the composite
\[
f(\mathbf{z}^v)=f(z^v_1,\ldots,z^v_m)
\]
lies in $\ZZ[u_1,\ldots,u_n]$ for every vertex $v$ of $P$.
We denote the set of all such polynomials by $\mathscr{Z}_{(P,\lambda)}$.
\end{definition}

Note that $\mathscr{Z}_{(P, \lambda)}$ forms a subring of $\ZZ[x_1, \dots, x_m]$, and
the ideal $\mathcal{I}$ in~\eqref{eq_ideal I} is also an ideal of $\mathscr{Z}_{(P,\lambda)}$. Indeed, for facets $F_{k_1}, \dots, F_{k_\ell}$ with $\bigcap_{j=1}^\ell F_{k_j}= \emptyset$, the polynomial $f(x_1, \dots, x_m)=\prod_{j=1}^\ell x_{k_j}$ satisfies $f(\mathbf{z}^v)=0$ for every vertex $v$ of $P$, and hence satisfies the integrality condition. 

\begin{definition}\label{def_wSR}
The \emph{weighted Stanley--Reisner ring} corresponding to a characteristic pair $(P, \lambda)$ is the quotient ring
\[
\wSR(P, \lambda) \colonequals \mathscr{Z}_{(P, \lambda)} /  \mathcal{I}
\]
where $\mathcal{I}$ is the ideal given in~\eqref{eq_ideal I}.
\end{definition}

\begin{proposition}(See \cite[Section 5]{BSS} or \cite[Theorem 4.11]{DKS})
\label{prop_eqiv_cohom_wSR}
For toric orbifolds $X(P, \lambda)$ with vanishing odd degree integral cohomology groups, there is a ring isomorphism 
\[
H^\ast_{T^n}(X(P, \lambda);\ZZ) \cong \wSR(P, \lambda).
\]
\end{proposition}

Regarding ordinary integral cohomology, if $H^*(X(P,\lambda);\ZZ)$ is concentrated in even degrees, the Serre spectral sequence of the Borel fibration 
\[
X(P,\lambda) \overset{\iota}{\hookrightarrow} ET^n\times_{T^n} X(P,\lambda) \xrightarrow{\pi} BT^n
\]
collapses at $E_2$-page. Hence, the map $\iota^\ast$ induces an isomorphism 
\[
H^{\ast}_{T^n}(X(P, \lambda);\ZZ)\otimes_{H^\ast(BT^n;\ZZ)} \ZZ \to H^\ast(X(P, \lambda);\ZZ).
\]
One can refer to \cite[theorem 1.1]{FrPu}. Hence, in this case, one can combine the above isomorphism together with Proposition \ref{prop_eqiv_cohom_wSR} to obtain a ring isomorphism 
\begin{equation}\label{eq_cohmgy_isom}
\Phi\colon H^*(X(P,\lambda);\ZZ)\to\wSR(P,\lambda) / \mathcal{J}
\end{equation}
where $\mathcal{J}$ is the ideal generated by monomials $\{y_1,\ldots,y_m\}$ defined by
\[
\begin{bmatrix} y_1 \\ \vdots \\ y_m \end{bmatrix} = 
\left[ \begin{array}{c|c|c} 
& & \\
\lambda_1 & \cdots &  \lambda_m \\
& & 
\end{array}\right]
\cdot \begin{bmatrix} x_1 \\ \vdots  \\ x_n \end{bmatrix}.
\]

As mentioned in Introduction, the computation $\wSR(P, \lambda)$ depends on explicit data from the characteristic pair and must be carried out individually in each case. Moreover, eliminating redundant generators modulo the ideal $\mathcal{J}$ makes the determination of a basis more involved. In the following section, we resolve this problem for the case of $4$-dimensional toric orbifolds.

\section{Intersection of lattices}\label{sec_lat_int}
For the rest of this paper, we focus on $4$-dimensional toric orbifolds. Let $P$ be a $2$-dimensional polygon and let $\lambda\colon \mathcal{F}(P) \to \ZZ^2$ be a characteristic function. Suppose that $P$ is an $m$-gon with $m\geq 3$, and let $F_1, \dots, F_m$ denote its edges. Write
\[
\lambda_i=(a_i, b_i)\colonequals\lambda(F_i)
\] for $i=1, \dots, m$. Then each $\lambda_i$ is primitive and $\det(\lambda_i,\lambda_{i+1})\neq 0$ for $i=1, \dots, m$, where the indices are taken modulo $m$, that is, $\lambda_{m+1}=\lambda_1$. See Figure \ref{fig_labeling}. 

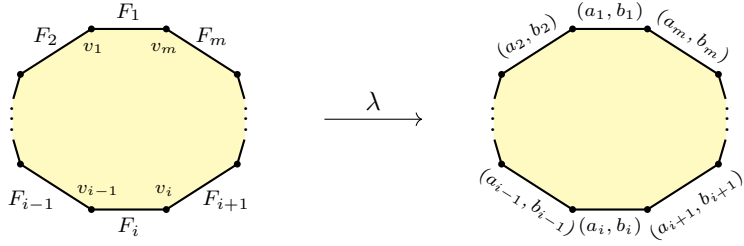
\begin{figure}
\begin{tikzpicture}[scale=1.3]
\coordinate (v0) at (10:1.2);
\coordinate (v00) at (350:1.2);

\coordinate (v1) at (0+22.5:1.2);
\coordinate (v2) at (45+22.5:1);
\coordinate (v3) at (90+22.5:1);
\coordinate (v4) at (135+22.5:1.2);
\coordinate (v41) at (170:1.2);

\coordinate (v50) at (190:1.2);
\coordinate (v5) at (180+22.5:1.2);
\coordinate (v6) at (225+22.5:1);
\coordinate (v7) at (270+22.5:1);
\coordinate (v8) at (315+22.5:1.2);

\draw[fill=yellow!30, yellow!30] (v0)--(v1)--(v2)--(v3)--(v4)--(v41)--(v50)--(v5)--(v6)--(v7)--(v8)--(v00)--cycle;

\draw[thick] (v0)--(v1)--(v2)--(v3)--(v4)--(v41);
\draw[thick] (v50)--(v5)--(v6)--(v7)--(v8)--(v00);

\node at (4:1.2) {$\vdots$};
\node at (176:1.2) {$\vdots$};

\foreach \v in {v1,v2,v3,v4,v5, v6,v7,v8}
    \fill (\v) circle (1pt);

\node at (45:1.2)  {\footnotesize$F_m$};
\node at (90:1.1)  {\footnotesize$F_1$};
\node at (135:1.2) {\footnotesize$F_2$};
\node at (220:1.3) {\footnotesize$F_{i-1}$};
\node at (270:1.1) {\footnotesize$F_i$};
\node at (320:1.3) {\footnotesize$F_{i+1}$};

\node[below] at (45+22.5:0.95) {\scriptsize$v_m$};
\node[below] at (90+22.5:0.95) {\scriptsize$v_1$};
\node[above] at (225+25.5:0.95) {\scriptsize$v_{i-1}$};
\node[above] at (270+22.5:0.95) {\scriptsize$v_i$};

\draw[->] (2,0)--(3,0);
\node[above] at (2.5,0) {$\lambda$};

\begin{scope}[xshift=140]
\coordinate (v0) at (10:1.2);
\coordinate (v00) at (350:1.2);

\coordinate (v1) at (0+22.5:1.2);
\coordinate (v2) at (45+22.5:1);
\coordinate (v3) at (90+22.5:1);
\coordinate (v4) at (135+22.5:1.2);
\coordinate (v41) at (170:1.2);

\coordinate (v50) at (190:1.2);
\coordinate (v5) at (180+22.5:1.2);
\coordinate (v6) at (225+22.5:1);
\coordinate (v7) at (270+22.5:1);
\coordinate (v8) at (315+22.5:1.2);

\draw[fill=yellow!30, yellow!30] (v0)--(v1)--(v2)--(v3)--(v4)--(v41)--(v50)--(v5)--(v6)--(v7)--(v8)--(v00)--cycle;

\draw[thick] (v0)--(v1)--(v2)--(v3)--(v4)--(v41);
\draw[thick] (v50)--(v5)--(v6)--(v7)--(v8)--(v00);

\node at (4:1.2) {$\vdots$};
\node at (176:1.2) {$\vdots$};

\foreach \v in {v1,v2,v3,v4,v5, v6,v7,v8}
    \fill (\v) circle (1pt);

\node[rotate=-32] at (45:1.2)  {\scriptsize$(a_m, b_m)$};
\node at (90:1.1)  {\scriptsize$(a_1, b_1)$};
\node[rotate=32] at (135:1.2) {\scriptsize$(a_2, b_2)$};
\node[rotate=-32] at (225:1.2) {\scriptsize$(a_{i-1}, b_{i-1})$};
\node at (270:1.1) {\scriptsize$(a_i, b_i)$};
\node[rotate=34] at (315:1.25) {\scriptsize$(a_{i+1}, b_{i+1})$};
\end{scope}

\end{tikzpicture}
\caption{Polygon and Characteristic function.}
\label{fig_labeling}
\end{figure}

In \cite[Section 5]{DKS}, the authors studied the integrality condition (Definition~\ref{def_int_cond}) for 2-dimensional characteristic pairs $(P, \lambda)$ and provided a more rigorous description of $\wSR(P, \lambda)$.

\begin{proposition}\cite[Theorem 5.2]{DKS}\label{prop_DKS}
Let $L_i$ be the free $\ZZ$-module  generated by row vectors of the following $(m\times m)$-matrix 
\begin{align*}
\Lambda_i & =
\left[ \begin{array}{c|cc|c}
I_{i-1} & &&\\ \hline 
& a_i & a_{i+1}&\\
& b_i& b_{i+1}&\\ \hline
& & & I_{m-i-1} 
\end{array}
\right] \text{   for }  i=1, \dots, m-1;\\
\Lambda_m &= \left[ \begin{array}{c|ccc|c}
a_1 && && a_m\\ \hline 
&&&& \\
& &I_{m-2}& & \\ 
&&&& \\ \hline 
b_1 & &&&  b_m 
\end{array} 
\right] \text{ for }i=m, 
%\end{split}
\end{align*}
where $I_k$ denotes the $(k\times k)$-identity matrix and all empty blocks are zeros. Then, the degree $2$ component of the weighted Stanley--Reisner ring $\wSR(P, \lambda)$ is given by 
\begin{equation}\label{eq_wSR2}
\wSR^2(P, \lambda) = 
\left\{ \sum_{i=1}^m c_i x_i ~ \middle|~  (c_1, \dots, c_m)\in \bigcap_{i=1}^m L_i \right\}. 
\end{equation}
\end{proposition}

Hence, we conclude the following corollary. 
\begin{corollary}\label{cor_KMZ}
For the case of $4$-dimensional toric orbifold $X(P, \lambda)$ associated with $(P, \lambda)$ described in Figure \ref{fig_labeling}, if it satisfies
\begin{equation}\label{eq_concentrated even deg condition}
\gcd\{ |a_ib_j-a_jb_i| \mid 1\leq i<j\leq m\}=1,
\end{equation}
then $H^2_{T^2}(X(P,\lambda);\ZZ)$ is isomorphic to $\wSR^2(P, \lambda)$ of \eqref{eq_wSR2} as $\ZZ$-modules. 
\end{corollary}
\begin{proof}
Note that every toric orbifold is simply connected and $H^3(X(P, \lambda);\ZZ)$ is isomorphic to $\ZZ^2/\spa_\ZZ \{(a_i, b_i)\mid i=1, \dots, m\}$. We refer to \cite{KMZ}. See also \cite{Jor, Fis} for relevant statements for toric varieties. The hypothesis \eqref{eq_concentrated even deg condition}
% $\gcd\{ |a_ib_j-a_jb_i| \mid 1\leq i<j\leq m\}=1$ 
implies that  $\spa_\ZZ\{(a_i, b_i)\mid i=1, \dots, m\}=\ZZ^2$, hence the cohomology $H^\ast(X(P, \lambda);\ZZ)$ is concentrated in even degrees. Now, the claim follows directly from Propositions~\ref{prop_eqiv_cohom_wSR} and \ref{prop_DKS}. 
\end{proof}

We now analyze the module structure of $H^2_{T^2}(X(P,\lambda);\ZZ)$ and construct an explicit basis. By Proposition~\ref{prop_DKS} and 
Corollary \ref{cor_KMZ} above, this problem reduces to studying the intersection of the lattices $L_1,\dots,L_m$. The main result of this section is formulated in Theorem~\ref{thm_main}.

For characteristic vectors $\lambda_i=(a_i, b_i)$  as above, we write 
\begin{equation}\label{eq_a_b}
\mathbf{a}\colonequals (a_1,\dots,a_m)\quad \text{and} \quad \mathbf{b}\colonequals (b_1,\dots,b_m)\in \mathbb{Z}^m,
\end{equation}
and regarding $\Lambda\colonequals \begin{bmatrix} \lambda_1 & \cdots & \lambda_m \end{bmatrix}$ as a linear map $\ZZ^m \to \ZZ^2$, let 
\begin{equation}\label{eq_K}
K\colonequals \ker(\Lambda)
=
\left\{
(t_1,\dots,t_m)\in \mathbb{Z}^m \;\middle|\;
\sum_{i=1}^m t_i a_i=0,\ \sum_{i=1}^m t_i b_i=0
\right\}.
\end{equation}
Note that $K$ is the space of linear relations of characteristic vectors $\lambda_1, \dots, \lambda_m$. 

Consider the linear operator 
\begin{equation}\label{eq_phi_linear_operator}
\phi\colon \ZZ^m \to \ZZ^m,~\phi(t_1, \dots, t_m)=(w_1, \dots, w_m)
\end{equation}
defined by $w_1=0$ and
\[ %\label{eq_linear_operator}
w_i:=\sum_{j=1}^{i-1}\det(\lambda_j,\lambda_i)t_j = \sum_{j=1}^{i-1}(a_jb_i-a_ib_j)t_j.
\]
for $i=2, \dots, m$.

\begin{lemma}\label{lem_w_m_0}
\begin{enumerate}
    \item If $(t_1, \dots, t_m)\in K$, then $w_m=0$. 
    \item The restriction $\phi|_K \colon K \to \ZZ^m$ is injective. 
\end{enumerate}
\end{lemma}
\begin{proof}
(1) For $(t_1, \dots, t_m)\in K$, we have $\sum_{j=1}^m t_j\lambda_j=\mathbf{0}$. 
Hence, we obtain  
\[
w_m:=\sum_{j=1}^{m-1}\det(\lambda_j,\lambda_m)t_j =\sum_{j=1}^{m}\det(\lambda_j,\lambda_m)t_j = \det \left( \sum_{j=1}^m t_j\lambda_j, \lambda_m\right)=0.  
\]

(2) Assume that $\phi(t_1, \dots, t_m)=(w_1, \dots, w_m)=(0,\dots, 0)$ for $(t_1, \dots, t_m)\in K$. Since $w_2=\det(\lambda_1, \lambda_2)t_1$ with  $\det(\lambda_1, \lambda_{2})\neq 0$, we have $t_1=0$. Next, we have 
\[
w_3=\det (\lambda_1, \lambda_3)t_1+\det (\lambda_2, \lambda_3)t_2=0,
\]
which implies that $t_2=0$ because $t_1=0$ and $\det (\lambda_2, \lambda_3)\neq 0$. Continuing this procedure and using the fact that $\det (\lambda_i, \lambda_{i+1})\neq 0$, we have $t_1=\cdots =t_{m-1}=0$. Finally, since $(t_1, \dots, t_m)\in K$ and $\lambda_m=(a_m, b_m)$ is a nonzero vector, we conclude that $t_m=0$. 
\end{proof}

\begin{theorem}\label{thm_main_lattice_intersection}
For the lattices $L_i$ defined above for $i=1, \dots, m$,  it follows that 
\[
L_1\cap \cdots \cap L_m = \mathbb{Z}\left<\mathbf{a}\right> \oplus \mathbb{Z}\left<\mathbf{b}\right>\oplus \phi(K).
\]
In particular, if $\{\kappa_1, \dots, \kappa_{m-2}\}$ is a $\mathbb{Z}$-basis of $K$, then
$\{ \mathbf{a}, \mathbf{b}, \phi(\kappa_1), \dots,  \phi(\kappa_{m-2})\}$
is a $\mathbb{Z}$-basis of $L_1\cap\cdots\cap L_m$.
\end{theorem}

\begin{proof}
We first claim that 
\begin{equation}\label{eq_int_equal_sum}
L_1\cap \cdots \cap L_m=\mathbb{Z}\left<\mathbf{a}\right> + \mathbb{Z}\left<\mathbf{b}\right>+ \phi(K).
\end{equation}
We start with the inclusion  $L_1\cap \cdots \cap L_m\subseteq\mathbb{Z}\left<\mathbf{a}\right> + \mathbb{Z}\left<\mathbf{b}\right>+ \phi(K)$.
Note that $\mathbf{x}=(x_1, \dots, x_m)\in L_i$ if and only if 
\begin{equation}\label{eq_coord_condition}
x_i=a_ip_i+b_iq_i,\quad \text{and} \quad x_{i+1}=a_{i+1}p_i+b_{i+1}q_i
\end{equation}
for some $p_i, q_i\in \ZZ$. Therefore, for $\mathbf{x}=(x_1, \dots, x_m)\in L_1\cap\cdots \cap L_m$, there exist
$p_1, q_1, p_2, q_2, \dots, p_m, q_m\in \ZZ$ 
such that \eqref{eq_coord_condition} holds for each $i=1, \dots, m$. Here, indices are taken modulo $m$, that is, $(a_{m+1}, b_{m+1})=(a_1, b_1)$. Hence, we get 
\[
p_ia_i+q_ib_i= p_{i-1}a_i + q_{i-1}b_i,
\]
for each $i=1, \dots, m$, where we set $(p_0, q_0)=(p_m, q_m)$. 

Since $\lambda_i=(a_i, b_i)$ is primitive for each $i=1,\ldots,m$, there exists $t_i\in\ZZ$ such that
\begin{equation}\label{eq_pq_ab}
(p_i-p_{i-1}, q_i- q_{i-1}) = t_i (-b_i, a_i)
\end{equation}
Summing Equation~\eqref{eq_pq_ab} over $i=1, \dots, m$, we get 
\[
0=\sum_{i=1}^m (p_i-p_{i-1})  = - \sum_{i=1}^m b_it_i \quad \text{and} \quad 
0=\sum_{i=1}^m (q_i-q_{i-1})  =  \sum_{i=1}^m a_it_i,
\]
implying that $(t_1, \dots, t_m)\in K$. 
Similarly, summing the first $i$ equations in~\eqref{eq_pq_ab} gives
\[
p_i=p_m - \sum_{j=1}^i b_jt_j \quad \text{and} \quad q_i=q_m + \sum_{j=1}^i a_jt_j.
\]
Therefore, we have for each $i=1, \dots, m$,
\begin{align*}
x_i&=p_ia_i+q_ib_i\\
&=\left(p_m - \sum_{j=1}^i b_jt_j\right) a_i + \left( q_m + \sum_{j=1}^i a_jt_j \right) b_i \\
&=p_ma_i+q_mb_i + \sum_{j=1}^i (a_jb_i - a_ib_j) t_j\\
&= p_ma_i+q_mb_i + \sum_{j=1}^{i-1} (a_jb_i - a_ib_j) t_j.
\end{align*}
This establishes that $\mathbf{x}=(x_1, \dots, x_m)\in \mathbb{Z}\left<\mathbf{a}\right> + \mathbb{Z}\left<\mathbf{b}\right>+ \phi(K)$. 

We now prove the reverse inclusion $\mathbb{Z}\left<\mathbf{a}\right> + \mathbb{Z}\left<\mathbf{b}\right>+ \phi(K)\subseteq L_1\cap \cdots \cap L_m$. For an element $\alpha \mathbf{a} + \beta \mathbf{b} + \phi(t)$ with $\alpha, \beta\in \mathbb{Z}$ and $t=(t_1, \dots, t_m) \in K$, we set 
\[
p_i=\alpha - \sum_{j=1}^i b_jt_j \quad \text{and} \quad q_i=\beta + \sum_{j=1}^i a_jt_j.
\]
Observe that $p_m= \alpha$ and $q_m= \beta$ as $t\in K$.
Then, writing $x_i$ as the $i$-th coordinate of $\alpha \mathbf{a} + \beta \mathbf{b} + \phi(t)$, we have 
\[
x_i = \alpha a_i + \beta b_i + \sum_{j=1}^{i-1}(a_jb_i-a_ib_j)t_j = \alpha a_i + \beta b_i + \sum_{j=1}^{i}(a_jb_i-a_ib_j)t_j  =  a_ip_i+b_iq_i
\]
for $i=1,\dots, m-1$. Furthermore, it follows that 
\begin{equation*}
% \label{eq_m_coor}
\begin{split}
x_m& =\alpha a_m + \beta b_m + \sum_{j=1}^{m-1}(a_jb_m-a_mb_j)t_j\\
&= \alpha a_m + \beta b_m + \sum_{j=1}^{m-1} \det (\lambda_j, \lambda_m)t_j  \\
%&= \alpha a_m + \beta b_m + \det ( t_1\lambda_1 + \cdots + t_{m-1}\lambda_{m-1}, \lambda_m)\\
&= \alpha a_m + \beta b_m 
\end{split}
\end{equation*}
where the last equality follows from Lemma \ref{lem_w_m_0}-(1). A similar computation gives the second equation in \eqref{eq_coord_condition}. Therefore
$\alpha \mathbf{a} + \beta \mathbf{b} + \phi(t)\in L_1 \cap \cdots \cap L_m$.
% and hence Equation~\eqref{eq_int_equal_sum} holds.

It remains to show that the right hand side of \eqref{eq_int_equal_sum} is indeed a direct sum $\ZZ\langle\textbf{a}\rangle\oplus\ZZ\langle\textbf{b}\rangle\oplus\phi(K)$. 
Assume that $\alpha \mathbf{a} + \beta \mathbf{b} + \phi(t)=0$ for some $\alpha, \beta\in \ZZ$ and $t\in K$. Note that if $\phi(t)=(w_1, \dots, w_m)$, then $w_1=0$ by convention and $w_m=0$ by Lemma \ref{lem_w_m_0}-(1). Then, we have 
\[
\alpha a_1+ \beta b_1 = \alpha a_m + \beta b_m = 0
\]
which implies that $\alpha=\beta=0$ because $\lambda_1=(a_1, b_1)$ and $\lambda_m=(a_m, b_m)$ are linearly independent. Hence, the result follows.

Finally, the second claim follows directly from Lemma \ref{lem_w_m_0}-(2). This completes the proof.
\end{proof}

\begin{theorem}\label{thm_main}
Let $X(P, \lambda)$ be a $4$-dimensional toric orbifold satisfying Equation~\eqref{eq_concentrated even deg condition}.
Then there is an isomorphism of $\ZZ$-modules
\begin{equation}\label{eq_wSR2_decomp}
\Psi\colon
\mathbb{Z}\left<\mathbf{a}\right> \oplus \mathbb{Z}\left<\mathbf{b}\right>\oplus \phi(K)\to
H^2_{T^2}(X(P, \lambda);\ZZ)\cong\wSR^2(P,\lambda)
\end{equation}
given by $\Psi(t_1,\ldots,t_m)=t_1x_1+\cdots+t_mx_m$.
Moreover, if $\{\kappa_1,\ldots,\kappa_{m-2}\}$ forms a basis of~$K$, then
\[
\{ \Psi(\mathbf{a}), \Psi(\mathbf{b}), (\Psi\circ\phi)(\kappa_1), \dots,  (\Psi\circ\phi)(\kappa_{m-2})\}
\]
forms a basis of $H^2_{T^2}(X(P,\lambda);\ZZ)$.
\end{theorem}
\begin{proof}
The proof follows directly from Theorems \ref{prop_DKS} and \ref{thm_main}
\end{proof}

In practice, one can compute a basis of $K$ using the Smith normal form or the Hermite normal form. The following example illustrates this calculation.

\begin{example}\label{ex_snf}
Consider the toric orbifold $X(P,\lambda)$ associated with a $4$-gon $P$ and the characteristic function $\lambda$ on $P$  defined by $\lambda_1=(-2,1),\lambda_2=(1,-2), \lambda_3=(2,1)$ and $\lambda_4=(1,2)$. For the matrix 
\[
\Lambda = \begin{bmatrix}
-2 & 1 & 2 & 1 \\
1 & -2 & 1 & 2
\end{bmatrix},
\]
the Smith normal form of $\Lambda$ is given by 
\[
U\Lambda V=\begin{bmatrix} 1&0&0&0 \\ 0 & 1 & 0 & 0\end{bmatrix}
\]
where one can take unimodular matrices $U\in \GL_2(\ZZ)$ and $V\in \GL_4(\ZZ)$ as
\[
U=\begin{bmatrix} 0 & 1 \\ 1 & 2\end{bmatrix} \quad \text{and} \quad
V=\begin{bmatrix} 1&  1&  5&  7\\
0 & 1&  4&  5 \\
0 & 1&  3&  5 \\
0&  0&  0& -1
\end{bmatrix}.
\]
Then, the last two column vectors of $V$ form a basis of $K=\ker (\Lambda\colon \ZZ^4 \to \ZZ^2)$. Finally, Theorem  \ref{thm_main_lattice_intersection} gives us the following basis of $L_1\cap \cdots \cap L_5$:
\[
\{(-2,1,2,1), (1,-2,1,2),(0, 15, 0,0), (0,21,-3,0)\}.
\]
Therefore, by Theorem \ref{thm_main}, we obtain a basis of $H^2_{T^2}(X(P, \lambda);\ZZ)\cong \wSR^2(P, \lambda)$ as follows
\[
\{-2x_1+x_2+x_3+x_4, x_1-2x_2+x_3+2x_4, 15x_2, 21x_2-3x_3\}.
\]
\end{example}

\section{Applications}\label{sect_applications}

\subsection{Construction of algebraic cellular bases}\label{subsec_alg_cel_basis}
In \cite{FSS2}, the authors studied the integral cohomology of a $4$-dimensional toric orbifold $X(P, \lambda)$ such that $\lambda$ satisfies 
\begin{equation}\label{eq_smooth pt Lambda}
\lambda_{m-1}=(1,0) \quad \text{and} \quad \lambda_m=(0,1).
\end{equation}
This condition is equivalent to the existence of a smooth fixed point in $X(P, \lambda)$. In this case, $X(P, \lambda)$ admits a cellular structure 
\begin{equation}\label{eq_CW}
X(P, \lambda) \simeq \left(\bigvee_{i=1}^{m-2} S^2\right) \cup e^4,
\end{equation}
which gives rise to cellular cohomology classes in $H^\ast(X(P, \lambda);\ZZ)$. Moreover, since the cellular structure \eqref{eq_CW} consists of even dimensional cells only, their corresponding cohomology classes form a basis. 
The authors then identified these classes, via the isomorphism \eqref{eq_cohmgy_isom}, with elements of $\wSR(P, \lambda)/\mathcal{J}$. The resulting basis is called the \emph{algebraic cellular basis} for $H^\ast(X(P, \lambda);\ZZ)$. 

\begin{definition}\label{dfn_alg cell basis}
Let $X(P,\lambda)$ be a $4$-dimensional toric orbifold satisfying~\eqref{eq_smooth pt Lambda}. Its \emph{algebraic cellular basis} is the set
$\{u_1,\ldots,u_{m-2};v\}$
of cohomology classes in $H^*(X(P,\lambda);\ZZ)$ defined by 
\begin{itemize}
\item
$u_i=\Phi^{-1}\left(\sum^{i}_{k=1}a_kb_ix_k+\sum^{m-2}_{k=i+1}a_ib_kx_k\right)$
\item
$v=\Phi^{-1}(x_{m-1}x_m)$,
\end{itemize}
where $\Phi\colon H^\ast(X(P, \lambda);\ZZ) \to \wSR(P, \lambda)/\mathcal{J}$ is the isomorphism in \eqref{eq_cohmgy_isom}.
\end{definition}

We note that the authors of \cite{FSS2} introduced the notion of a toric morphism between toric orbifolds and use this to justify that the algebraic cellular basis defined above is indeed a basis for $\widetilde{H}^\ast(X(P, \lambda);\ZZ)$. 
In particular, the well-definedness of degree-two elements, namely, showing that the polynomials defining $u_i$ are indeed elements of $\wSR(P, \lambda)$, requires series of technical calculations of toric morphisms. Using Theorem~\ref{thm_main}, we now give a short alternate proof regarding the degree-two elements.

\begin{theorem}\label{thm_alg cellular basis}
Let $X(P,\lambda)$ be a $4$-dimensional toric orbifold satisfying~\eqref{eq_smooth pt Lambda}.
Then the degree two elements $u_1,\ldots,u_{m-2}$ in Definition~\ref{dfn_alg cell basis} are well-defined and form a basis for $H^2(X(P,\lambda);\ZZ)$.
\end{theorem}

\begin{proof}
As the characteristic function $\lambda$ satisfies \eqref{eq_smooth pt Lambda}, one can simply take a basis of $K$ as $\{\xi_1, \dots, \xi_{m-2}\}\in \ZZ^m$ given by  
\begin{equation}\label{eq_xi_i}
\xi_i\colonequals (0,\dots, 0, \underset{\substack{\uparrow \\ i\text{-th}}}{1}, 0, \dots, 0, -a_i, -b_i),
\end{equation}
Then, by Theorem \ref{thm_main}, the set 
\[
\{\Psi(\mathbf{a}), \Psi(\mathbf{b}), (\Psi\circ\phi)(\xi_1), \dots, (\Psi\circ \phi)(\xi_{m-2})\}
\]
with $\mathbf{a}=(a_1, \dots, a_{m-2}, 1,0)$ and $\mathbf{b}=(b_1, \dots, b_{m-2}, 0,1)$ forms a basis of $\wSR(P, \lambda)$. Since $\Psi(\mathbf{a}), \Psi(\mathbf{b})$ are generators of the ideal $\mathcal{J}$, the remaining elements form a basis of $\wSR^2(P, \lambda)/\mathcal{J}\cong H^2(X(P, \lambda);\ZZ)$. 

We complete the proof by showing that $(\Psi\circ\phi)(\xi_i)=\Phi(u_i)$ for $i=1, \dots, m-2$. Indeed, by the definition $\phi$ given in \eqref{eq_phi_linear_operator} together with \eqref{eq_xi_i}, we have 
\begin{align*}
(\Psi\circ\phi)(\mu_i) &=  \sum_{k={i+1}}^{m} \det(\lambda_i, \lambda_k)x_k -a_i \det(\lambda_{m-1}, \lambda_m)x_m\\
&= \sum_{k={i+1}}^{m-2} (a_ib_k-a_kb_i)x_k - b_ix_{m-1} \\
&= \sum_{k={i+1}}^{m-2} (a_ib_k-a_kb_i)x_k + b_i\sum_{k=1}^{m-2}a_kx_k\\
&=\sum_{k=1}^ia_kb_ix_k + \sum_{k=i+1}^{m-2}a_ib_k x_k,
\end{align*}
where the second equality follows from \eqref{eq_smooth pt Lambda} and the third one also holds because $\Psi(\mathbf{a})=\sum_{k=1}^{m-2}a_kx_k+x_{m-1}\in \mathcal{J}$. Hence, we have established the claim. 
\end{proof}

\begin{example}
Let $X(P,\lambda)$ be the $4$-dimensional toric orbifold associated to a $4$-gon $P$ and the characteristic function $\lambda$ defined by $\lambda_1=(-2,1), \lambda_2=(1,-2), \lambda_3=(1,0)$ and $\lambda_4=(0,1)$. 
Following the proof of Theorem \ref{thm_alg cellular basis}, one can take 
\[
\xi_1=(1,0,2,-1) \quad \text{and} \quad \xi_2=(0,1,-1,2)
\]
as a basis of $K$. Then, Theorem \ref{thm_main} implies that 
\begin{equation}\label{eq_example_gen}
\{\Psi(\mathbf{a}), \Psi(\mathbf{b}), \Psi(\xi_1), \Psi(\xi_2)\}=
\{-2x_1+x_2+x_3, x_1-2x_2+x_4, 3x_2-x_3, 2x_3\}.
\end{equation}
for $\mathbf{a}=(-2,1,1,0)$ and $\mathbf{b}=(1,-2,0,1)$ is a basis of $\wSR^2(P, \lambda)$. Applying the linear ideal $\mathcal{J}$ generated by the first two elements of \eqref{eq_example_gen}, observe that 
\[
[3x_2-x_3]=[-2x_1+4x_2] \quad \text{and} \quad [2x_3]= [4x_1-2x_2],
\]
whose preimages via the map $\Phi$ of \eqref{eq_cohmgy_isom} agree with the algebraic cellular basis of Definition \ref{dfn_alg cell basis} for  $X(P, \lambda)$.
\end{example}

\subsection{Description of Cartier divisors on toric surfaces}
As discussed in Remark \ref{rmk_toric_var}, if the characteristic pair $(P, \lambda)$ is defined from a rational polytope in $M\otimes_\ZZ \RR$, where $M$ is the character lattice of $T^n$, then $X(P, \lambda)$ can also be regarded as a topological model of the toric variety $X_\Sigma$ associated with the normal fan $\Sigma$ of~$P$.  

To be more precise, there is a bijection map between the set $\Sigma^{(1)}$ of $1$-dimensional cones and the set $\mathcal{F}(P)$ of facets. Each $\rho_i\in \Sigma^{(1)}$ is the cone generated by $\lambda(F_i)\in \ZZ^n$ for the facet $F_i$ corresponding to~$\rho_i$, where we identify $\ZZ^n$ with the lattice $N$ of 1-parameter subgroups of $T^n$. In addition, the set $\{\rho_1, \dots, \rho_k\}$ forms a $k$-dimensional cone in $\Sigma$ if and only if the corresponding facets $\{F_1, \dots, F_k\}$ have a non-empty intersection. Then, the isomorphism \eqref{eq_rational_cohom_isom} is defined by sending each torus invariant Weil divisor $D_{\rho_i}$ to the generator  $x_i$ corresponding to $F_i\in \mathcal{F}(P)$.

In this case, Theorem~\ref{thm_main} gives a complete description of the group $\CDiv_T(X_\Sigma)$ of torus-invariant Cartier divisors and the Picard group $\Pic(X_\Sigma)$. Also, we obtain the index of $\Pic(X_\Sigma)$ in the class group $\clas(X_{\Sigma})$, which will be discussed in Corollary~\ref{cor} below.  The definitions and properties of these groups can be found in \cite[Chapter 4]{CLS}.

The following proposition is well-known. See, for instance, \cite[Theorem 4.2.8]{CLS} or \cite[Section 3.3]{Ful}.

\begin{proposition}\label{prop_Cdiv}
A Weil divisor $\sum_{\rho\in \Sigma^{(1)}}a_\rho D_\rho$ is a Cartier divisor if and only if for each maximal cone $\sigma$, there is $m_\sigma\in M$ such that $\left< m_\sigma, \lambda(F_\rho)\right>=-a_\rho$ for all $\rho\in \Sigma^{(1)}$. 
\end{proposition}

In fact, it is straightforward to see that the condition in Proposition \ref{prop_Cdiv} coincides with the integrality condition in Definition \ref{def_int_cond} for homogeneous polynomials of degree $2$ in $\SR(P)$, using the correspondence between $D_{\rho_i}$ and $x_i\in \SR(P)$. Hence, we conclude the following corollary. 

\begin{corollary}\label{cor}
Let $X_{\Sigma}$ be a toric variety of real dimension $4$ with $H^3(X_{\Sigma};\ZZ)=0$, and let $(P,\lambda)$ be its associated characteristic pair. Then, the following statements hold.
\begin{enumerate}
\item The group $\CDiv_T(X_\Sigma)$ of torus-invariant Cartier divisors is isomorphic to~$\wSR^2(P, \lambda)$ of \eqref{eq_wSR2} as $\ZZ$-modules.
\item The Picard group $\Pic(X_\Sigma)$ is isomorphic to $\phi(K)$ of \eqref{eq_wSR2_decomp} as $\ZZ$-modules. 
\item The index of $\Pic(X_\Sigma)$ in the class group $\clas(X_\Sigma)$ is
\[
\det\begin{bmatrix}
\textbf{a}^t	&\textbf{b}^t	&\phi(\kappa_1)^t	&\cdots	&\phi(\kappa_{m-2})^t
\end{bmatrix}
\]
for some basis  $\{\kappa_1,\ldots,\kappa_{m-2}\}$ of $K$.
\end{enumerate}
\end{corollary}
\begin{proof}
The proof of (1) is straightforward, as explained above. For~(2), recall from the proof of Corollary \ref{cor_KMZ} that the assumption $H^3(X_\Sigma;\ZZ)=0$ is equivalent to the condition that $\spa_{\ZZ} \{ (a_1, b_1), \dots, (a_m, b_m)\}=\ZZ^2$. Hence, we have the following diagram of exact sequences: 
\begin{equation}\label{eq_diagram}
\begin{tikzcd}
0\rar & M\arrow{d}{\cong}  \rar & \CDiv_T(X_\Sigma) \rar \arrow{d}{\cong}&  \Pic(X_\Sigma) \rar & 0\\
0 \rar & \ZZ^2 \rar & \mathbb{Z}\left<\mathbf{a}\right> \oplus \mathbb{Z}\left<\mathbf{b}\right>\oplus \phi(K) \rar &  \phi(K) \rar & 0,
\end{tikzcd}
\end{equation}
where the horizontal map $M \to \CDiv_T(X_\Sigma)$ is given by $m\mapsto \sum_{\rho\in \Sigma^{(1)}} \left< m, \lambda(F_\rho) \right>D_\rho$.  See \cite[Theorem 4.2.1]{CLS}) for the upper exact sequence. 
The commutativity of the left square induces a map  $\Pic(X_{\Sigma})\to\phi(K)$, which is an isomorphism by the Five Lemma. 

To prove~(3), consider the short exact sequence 
\begin{equation}\label{eq_ses_div_cl}
\begin{tikzcd}
    0\rar & M \rar&  \Div_{T}(X_\Sigma) \rar & \clas(X_\Sigma) \rar & 0
\end{tikzcd}
\end{equation}
for the group $\Div_{T}(X_\Sigma)=\bigoplus_{\rho\in \Sigma^{(1)}}\ZZ\left<D_\rho \right>$ of Weil divisors. 
One can refer to \cite[Theorem 4.1.3]{CLS} for the exactness. Combining \eqref{eq_ses_div_cl} with \eqref{eq_diagram} concludes that 
\[
[\clas(X_\Sigma): \Pic(X_\Sigma)] = [\Div_T(X_\Sigma): \CDiv_T(X_\Sigma)] =  [\ZZ^m: \mathbb{Z}\left<\mathbf{a}\right> \oplus \mathbb{Z}\left<\mathbf{b}\right>\oplus \phi(K)].
\]
Hence, we have established the claim. 
\end{proof}

%
%\bibliographystyle{amsalpha}
%\bibliography{bibliography}
%
%\providecommand{\bysame}{\leavevmode\hbox to3em{\hrulefill}\thinspace}
%\providecommand{\MR}{\relax\ifhmode\unskip\space\fi MR }
%% \MRhref is called by the amsart/book/proc definition of \MR.
%\providecommand{\MRhref}[2]{%
%  \href{http://www.ams.org/mathscinet-getitem?mr=#1}{#2}
%}

\providecommand{\href}[2]{#2}

\end{document}